\documentclass[12pt, reqno]{amsart}
\usepackage{ amsmath,amsthm, amscd, amsfonts, amssymb, graphicx, color}
\usepackage[bookmarksnumbered, colorlinks, plainpages,linkcolor=red,citecolor=blue]{hyperref}
\newtheorem{thm}{Theorem}[section]
\newtheorem{cor}[thm]{Corollary}

\newtheorem{exam}[thm]{Example}
\numberwithin{equation}{section}

\begin{document}

\title{the group invertibility of matrices over B\'ezout domains}

\author{Dayong Liu$^*$}
\author{Aixiang Fang}
\address{
College of Science \\ Central South University of Forestry and Technology  \\ Changsha , China}
\email{<liudy@csuft.edu.cn>}
\address{
College of Mathematics and Physics \\ Hunan University of Arts and Science \\ Changde,  China}
\email{<fangaixiangwenli@163.com>}


\begin{abstract}
Let $R$ be a B\'ezout domain, and let $A,B,C\in R^{n\times n}$ with $ABA=ACA$. If $AB$ and $CA$ are group invertible, we prove that $AB$ is similar to $CA$.
Moreover, we have $(AB)^{\#}$ is similar to $(CA)^{\#}$. This generalize the main result of Cao and Li(Group inverses for matrices over a B\'ezout domain, {\it Electronic J. Linear Algebra}, {\bf 18}(2009), 600--612). \end{abstract}

\thanks{Corresponding author: Dayong Liu}

\maketitle

\section{Introduction}

Let $R$ be a associate ring with an identity 1. An element $a\in R$ is Drazin invertible if there exists an element $x\in R$
such that
$$ ax=xa , \ \ xax=x , \ \ a^{k+1}x=a^{k}$$
for some nonnegative integer $k$. The preceding $x$ is unique if it exists. The element $x$ is called Drazin inverse of $a$, and denote $x$ by $a^D$.
The least nonnegative integer $k$ satisfying these equations above is the Drazin index $ind(a)$ of $a$. When $ind(a)=1$, we say that $a$ is group invertible, that is
$$ax=xa, x=xax, a=axa.$$
Denote $x$ by $a^{\#}$. We use
$R^{\#}$ to stand for the set of all group invertible elements in $R$.

Two elements $a, b \in R$  are similar, written $a\sim b$, if there exists an invertible element $s$ such that $a=s^{-1}bs$.

It is important for Similarity to solute linear equations. However, it is difficulty to consider similarity for usual ring.
In view of this, we will investigate the similarity over B\'ezout domain. An integral domain is called a B\'ezout domain if every its finitely generated ideal is principal.
Some authors have discussed the similarity of products of matrices on background of generalized inverse. For example,
in \cite[Theorem 3.6]{CL}, Cao and Li proved that $(AB)^{\#}$ and $(BA)^{\#}$ exist implies that $AB\sim BA$ for any $n\times n$ matrices $A$ and $B$ over a B\'ezout domain $R$.

This paper is partially motivated by Flanders' a classic theorem which states that the elementary divisors of $AB$ which do not have zero as a root coincide with those of $BA$(see \cite[Theorem 2]{F}).
In \cite[Theorem 2]{HARTWIG}, for strongly $\pi$-regular ring $R$, Hartwig proved that $R$ is unit-regular if and only if $R$ is regular and $(ab)^D\sim (ba)^D$ for every $a, b\in R$.

In this paper, we shall improve Flanders, Cao and Li's results.
In Section 2, we present a new characterization of the similarity. Under the condition $ABA=ACA$, we prove that $AB$ and $CA$ are similar if $AB$ and $CA$ are both group invertible for any $n\times n$ matrices $A, B$ and $C$ over a B\'ezout domain $R$. An example is provided to illustrate that our result is a nontrivial generalization of \cite[Theorem 3.6]{CL}. In Section 3, we turn to the similarity of the group inverses of products of matrices over B\'ezout domains. It is shown that $(AB)^{\#}$ and $(CA)^{\#}$ are similar if $AB$ and $CA$ are both group invertible for any $n\times n$ matrices $A, B$ and $C$ over a B\'ezout domain $R$.

Throughout this paper, $R$ is a B\'ezout domain. $R^{m\times n}$ denotes the set of all $m\times n$ matrices over $R$. Let $A\in R^{m\times n}$,
then $R_{r}(A)$ is the space spanned by the columns of $A$:
$$R_{r}(A)=\{Ax|x\in R^{n\times 1}\}\subseteq R^{m\times 1}.$$
Likewise, $R_{l}(A)$ is the space spanned by the rows of $A$:
$$R_{l}(A)=\{yA|y\in R^{1\times m}\}\subseteq R^{1\times n}.$$
The column rank (respectively, row rank) of $A$ is defined as the dimension of $R_{r}(A)$(respectively, $R_{l}(A)$).
The inner rank of $A$ is determined as the least $r$ such that $A=BC$, where $B\in R^{m\times r}, C\in R^{r\times n}$.
Over a B\'ezout domain, it is proved that the column rank, row rank and inner rank of any matrix $A$ coincide with each other(see \cite[Proposition 2.3.4]{HUANG} and \cite[Page 245]{COHN}), the common number is called the rank of $A$, written $rank(A)$. Define the rank of a zero matrix is 0. The rank of $n\times n$ non-singular matrix over $R$ is $n$.
An $n\times n$ matrix $A$ over a B\'ezout domain $R$ is invertible in case $AB=I_{n}=BA$ for an $n\times n$ matrix $B$. Obviously, $B$ is unique if it exists, denote $A^{-1}$.

\section{the similarity of product of matrices}

The purpose of this section is concerned with the similarity of product of matrices over B\'ezout domains. We now derive

\begin{thm}\label{thm1}
Let $R$ be a B\'ezout domain, and let $A,B,C\in R^{n\times n}$ with $ABA=ACA$. If $AB$ and $CA$ are group invertible,
then $AB$ is similar to $CA$.
\end{thm}
\begin{proof}
For $A,B,C\in R^{n\times n}$, by \cite[Lemma 3.5]{CL}, there exist invertible matrices $P$ and $Q$ such that
$$A=P\begin{bmatrix}
       A_1 & 0 \\
       0   & 0 \\
     \end{bmatrix}Q   \ and \
B=Q^{-1}\begin{bmatrix}
          B_1 & B_2 & B_3 \\
          0   & 0   & 0   \\
          C_1 & C_2 & C_3 \\
          0   & 0   & 0   \\
        \end{bmatrix}P^{-1},$$
also there exists an invertible matrix $Q'$ together with the invertible matrix $P$ above such that
$$A=P\begin{bmatrix}
       A'_1 & 0 \\
       0   & 0 \\
     \end{bmatrix}Q'   \ and \
C=Q'^{-1}\begin{bmatrix}
          B'_1 & B'_2 & B'_3 \\
          0   & 0   & 0   \\
          C'_1 & C'_2 & C'_3 \\
          0   & 0   & 0   \\
        \end{bmatrix}P^{-1},$$
where $A_1, A'_1 \in R^{r\times r}$, $Q=diag(Q_1, Q_2)N$, $Q'=diag(Q'_1, Q'_2)N$, $A_1Q_1=\Delta=A'_1Q'_1\in R^{r\times r}$, $rank\Delta=rankA_1=rankA'_1=rankA=r$,
$Q_1, Q_2, Q'_1, Q'_2$ and $N$ are invertible, $B_1, B'_1\in R^{s\times s}$, $C_1, C'_1\in R^{t\times s}$.
Then
$$AB=P\begin{bmatrix}
\vspace{2mm}
        A_1\begin{bmatrix}
             B_1 & B_2 \\
             0   & 0   \\
           \end{bmatrix}   &    A_1\begin{bmatrix}
                                           B_3    \\
                                           0      \\
                                   \end{bmatrix}   \\
            0              &    0                  \\
      \end{bmatrix}P^{-1}$$
and
$$CA=Q'^{-1}\begin{bmatrix}
\vspace{2mm}
               \begin{bmatrix}
                 B'_1 & B'_2 \\
                 0   & 0   \\
               \end{bmatrix}A'_1   &   0      \\
               \begin{bmatrix}
                 C'_1 & C'_2 \\
                 0   & 0   \\
               \end{bmatrix}A'_1   &   0      \\
      \end{bmatrix}Q'.$$
By \cite[Corollary 2.4 and 2.5]{CL}, we get
\begin{center}
$(AB)^{\#}$ exists $\Leftrightarrow$
$\left(A_1\begin{bmatrix}
             B_1 & B_2 \\
             0   & 0   \\
           \end{bmatrix}\right)^{\#}$ exists, and
\end{center}
\begin{center}
$R_{r}\left(A_1\begin{bmatrix}
                B_3    \\
                 0      \\
              \end{bmatrix}\right)\subseteq
R_{r}\left(A_1\begin{bmatrix}
             B_1 & B_2 \\
             0   & 0   \\
           \end{bmatrix}\right).$
\end{center}
\begin{center}
$(CA)^{\#}$ exists $\Leftrightarrow$
$\left(\begin{bmatrix}
        B'_1 & B'_2 \\
        0   & 0   \\
        \end{bmatrix}A'_1\right)^{\#}$ exists, and
\end{center}
\begin{center}
$R_{l}\left(\begin{bmatrix}
             C'_1 & C'_2 \\
              0   & 0   \\
            \end{bmatrix}A'_1\right)\subseteq
R_{l}\left(\begin{bmatrix}
           B'_1 & B'_2 \\
            0   & 0   \\
           \end{bmatrix}A'_1\right).$
\end{center}
Therefore, there exist $E, F$ such that
$$A_1\begin{bmatrix}
                B_3    \\
                 0      \\
              \end{bmatrix}
       =A_1\begin{bmatrix}
             B_1 & B_2 \\
             0   & 0   \\
           \end{bmatrix}E,
\begin{bmatrix}
C'_1 & C'_2 \\
 0   & 0   \\
\end{bmatrix}A'_1
=F\begin{bmatrix}
  B'_1 & B'_2 \\
   0   & 0   \\
\end{bmatrix}A'_1.$$
Then
$$AB=P
\begin{bmatrix}
  I & -E \\
  0 &  I \\
\end{bmatrix}
\begin{bmatrix}
\vspace{2mm}
        A_1\begin{bmatrix}
             B_1 & B_2 \\
             0   & 0   \\
           \end{bmatrix}   &    0   \\
            0              &    0   \\
      \end{bmatrix}
\begin{bmatrix}
  I &  E \\
  0 &  I \\
\end{bmatrix}P^{-1}, $$
and
$$CA=Q'^{-1}
\begin{bmatrix}
  I & 0 \\
  F & I \\
\end{bmatrix}
\begin{bmatrix}
\vspace{2mm}
\begin{bmatrix}
B'_1 & B'_2 \\
 0   & 0   \\
\end{bmatrix}A'_1   &    0   \\
            0       &    0   \\
      \end{bmatrix}
\begin{bmatrix}
  I &  0 \\
 -F &  I \\
\end{bmatrix}Q'. $$
Let
$$\begin{bmatrix}
B_1 & B_2 \\
 0  & 0   \\
\end{bmatrix}
=R\begin{bmatrix}
D & 0   \\
0 & 0   \\
\end{bmatrix}S, \
\begin{bmatrix}
B'_1 & B'_2 \\
 0   & 0   \\
\end{bmatrix}
=R'\begin{bmatrix}
D' & 0   \\
0 & 0   \\
\end{bmatrix}S'$$
for some invertible matrices $R, S, R', S'$, where $D\in R^{r_1\times r_1}$, $D'\in R^{s_1\times s_1}$, $rankD=r_1, rankD'=s_1$.
We obtain
$$A_1
\begin{bmatrix}
B_1 & B_2 \\
0   & 0   \\
\end{bmatrix}
= A_1 R
\begin{bmatrix}
D & 0   \\
0 & 0   \\
\end{bmatrix}S, \
\begin{bmatrix}
B'_1 & B'_2 \\
 0   & 0   \\
\end{bmatrix}A'_1
=
R'\begin{bmatrix}
D' & 0   \\
0 & 0   \\
\end{bmatrix}S'A'_1.$$
Let $SA_1 R=C=\begin{bmatrix}
                C_4 & C_5 \\
                C_6 & C_7 \\
              \end{bmatrix}$,
$S'A'_1 R'=C'=\begin{bmatrix}
                C'_4 & C'_5 \\
                C'_6 & C'_7 \\
              \end{bmatrix}$. Then $rankC=rankC'=r, $
and
$$A_1
\begin{bmatrix}
B_1 & B_2 \\
0   & 0   \\
\end{bmatrix}
=S^{-1}
\begin{bmatrix}
C_4 D & 0   \\
C_6 D & 0   \\
\end{bmatrix}S,$$
$$\begin{bmatrix}
B'_1 & B'_2 \\
 0   & 0   \\
\end{bmatrix}A'_1
=
R'\begin{bmatrix}
D'C'_4 & D'C'_5   \\
0      & 0   \\
\end{bmatrix}R'^{-1}.$$
Since $\left(A_1
\begin{bmatrix}
B_1 & B_2 \\
0   & 0   \\
\end{bmatrix}\right)^{\#}$ exists,
$\begin{bmatrix}
C_4 D & 0   \\
C_6 D & 0   \\
\end{bmatrix}$ is group invertible. Furthermore, $(C_4 D)^{\#}$ exists and $R_{l}(C_6 D)\subset R_{l}(C_4 D)$, then $C_{6}D=GC_4 D$ for some matrix $G$. Hence
$$\begin{bmatrix}
C_4 D & 0   \\
C_6 D & 0   \\
\end{bmatrix}
=
\begin{bmatrix}
I  &  0   \\
G  &  I   \\
\end{bmatrix}
\begin{bmatrix}
C_4 D & 0   \\
0     & 0   \\
\end{bmatrix}
\begin{bmatrix}
 I  &  0   \\
-G  &  I   \\
\end{bmatrix},$$
$$ rank(C_4)=rank(C_4 D)=r_1. $$
Therefore,
$$A_1
\begin{bmatrix}
B_1 & B_2 \\
0   & 0   \\
\end{bmatrix}
\sim
\begin{bmatrix}
C_4D & 0   \\
0    & 0   \\
\end{bmatrix}.$$
Analogously, $D'C'_4$ is group invertible, and
$$\begin{bmatrix}
B'_1 & B'_2 \\
 0   & 0   \\
\end{bmatrix}A'_1
\sim
\begin{bmatrix}
D'C'_4 & 0   \\
0      & 0   \\
\end{bmatrix},$$
$$ rank(C'_4)=rank(D'C'_4)=s_1. $$
By $(C_4 D)^{\#}$ exists and $rank(C_4)=rank(D)=r_1$, there exists $X$ such that
$C_4 DXC_4 D=C_4 D$, hence $DXC_4 D=D$ and $C_4 DXC_4=C_4$. Since $rank(C_4)=rank(D)=r_1$, we get $DXC_4=I=XC_4D$ and $C_4 DX=I=DXC_4$, which means $C_4$ and $D$ are invertible.
Thus $C_4D$ is invertible.
Likewise, $C'_4$, $D'$ and $D'C'_4$ are also invertible. \\
\indent
On the other hand, we have\\
$\begin{array}{rcl}
\vspace{2mm}
ABA  & = &
P
\begin{bmatrix}
  I & -E \\
  0 &  I \\
\end{bmatrix}
\begin{bmatrix}
        A_1\begin{bmatrix}
             B_1 & B_2 \\
             0   & 0   \\
           \end{bmatrix}   &    0   \\
            0              &    0   \\
      \end{bmatrix}
\begin{bmatrix}
  I &  E \\
  0 &  I \\
\end{bmatrix}P^{-1}
P\begin{bmatrix}
       A_1 & 0 \\
       0   & 0 \\
     \end{bmatrix}Q     \\
\vspace{2mm}
  &  =  &
P
\begin{bmatrix}
\vspace{2mm}
        A_1\begin{bmatrix}
             B_1 & B_2 \\
             0   & 0   \\
           \end{bmatrix}A_1   &    0   \\
            0                 &    0   \\
      \end{bmatrix}
      \begin{bmatrix}
             Q_1 &   0     \\
             0   &   Q_2   \\
           \end{bmatrix}N   \\
&   =   &
P
\begin{bmatrix}
\vspace{2mm}
        S^{-1}
        \begin{bmatrix}
             I   &    0   \\
             G   &    I   \\
        \end{bmatrix}
        \begin{bmatrix}
             C_4D   &    0   \\
             0      &    0   \\
        \end{bmatrix}
        \begin{bmatrix}
             I   &    0   \\
            -G   &    I   \\
        \end{bmatrix}SA_1Q_1     &    0   \\
                      0          &    0   \\
\end{bmatrix}N
\end{array}$   \\
and   \\
$\begin{array}{rcl}
\vspace{2mm}
ACA  &  =  &
P
\begin{bmatrix}
  A'_1 &  0 \\
  0    &  0 \\
\end{bmatrix}Q'
Q'^{-1}
\begin{bmatrix}
  I & 0 \\
  F & I \\
\end{bmatrix}
\begin{bmatrix}
\vspace{2mm}
\begin{bmatrix}
B'_1 & B'_2 \\
 0   & 0   \\
\end{bmatrix}A'_1   &    0   \\
            0       &    0   \\
      \end{bmatrix}
\begin{bmatrix}
  I &  0 \\
 -F &  I \\
\end{bmatrix}Q'   \\
\vspace{2mm}
    &   =   &
P
\begin{bmatrix}
\vspace{2mm}
        A'_1\begin{bmatrix}
             B'_1 & B'_2    \\
             0    & 0       \\
           \end{bmatrix}A'_1   &    0   \\
            0                  &    0   \\
      \end{bmatrix}
      \begin{bmatrix}
             Q'_1   &   0     \\
             0      &   Q'_2   \\
           \end{bmatrix}N   \\
\vspace{2mm}
   &   =   &
P
\begin{bmatrix}
\vspace{2mm}
        A'_1R'
           \begin{bmatrix}
             D'    &   0       \\
             0     &   0       \\
           \end{bmatrix}S'A'_1Q'_1   &    0   \\
            0                        &    0   \\
      \end{bmatrix}N   \\
\vspace{2mm}
   &   =   &
P
\begin{bmatrix}
\vspace{2mm}
        S'^{-1}
           \begin{bmatrix}
             C'_4D'    &   0       \\
             C'_6D'       &   0       \\
           \end{bmatrix}S'A'_1Q'_1   &    0   \\
            0                        &    0   \\
      \end{bmatrix}N
\end{array} $   \\
$\begin{array}{rcl}
   &   =   &
P
\begin{bmatrix}
\vspace{2mm}
        S'^{-1}
        \begin{bmatrix}
             I                 &   0       \\
             C'_6{C'_4}^{-1}     &   I       \\
           \end{bmatrix}
           \begin{bmatrix}
             C'_4D'    &   0       \\
             0         &   0       \\
           \end{bmatrix}
           \begin{bmatrix}
             I                  &   0       \\
             -C'_6{C'_4}^{-1}     &   I       \\
           \end{bmatrix}S'A'_1Q'_1   &    0   \\
            0                        &    0   \\
     \end{bmatrix}N.   \\
\end{array} $   \\
In view of $ABA=ACA$ and the invertibility of $P$ and $N$,  we get  \\
\noindent
$S^{-1}
        \begin{bmatrix}
             I   &    0   \\
             G   &    I   \\
        \end{bmatrix}
        \begin{bmatrix}
             C_4D   &    0   \\
             0      &    0   \\
        \end{bmatrix}
        \begin{bmatrix}
             I   &    0   \\
            -G   &    I   \\
        \end{bmatrix}S$   \\
\indent\hfill{
$=
S'^{-1}
        \begin{bmatrix}
             I                 &   0       \\
             C'_6{C'_4}^{-1}     &   I       \\
           \end{bmatrix}
           \begin{bmatrix}
             C'_4D'    &   0       \\
             0         &   0       \\
           \end{bmatrix}
           \begin{bmatrix}
             I                  &   0       \\
             -C'_6{C'_4}^{-1}     &   I       \\
           \end{bmatrix}S' . $}  \\
\noindent
It follows that
$$\begin{bmatrix}
C_4D   &    0   \\
0      &    0   \\
\end{bmatrix}
\sim
\begin{bmatrix}
C'_4D'   &    0   \\
0        &    0   \\
\end{bmatrix},$$
then
$$C_4D\sim C'_4D'=C'_4D'C'_4{C'_4}^{-1}.$$
Therefore, $$C_4 D\sim D'C'_4. $$
That is,
$$\begin{bmatrix}
C_4D   &    0   \\
0      &    0   \\
\end{bmatrix}
\sim
\begin{bmatrix}
D'C'_4   &    0   \\
0        &    0   \\
\end{bmatrix}.$$
We have the conclusion $AB\sim CA $.
\end{proof}

\begin{cor}\label{cor1}
Let $A, B, C\in R^{n\times n}$ and $ABA=ACA$. If $R_{r}(A)=R_{r}(ABA)$, then $AB\sim CA$.
\end{cor}
\begin{proof}
By $R_{r}(A)=R_{r}(ABA)$, $(AB)^{\#}$ exists. $R_{r}(A)=R_{r}(ACA)$ implies $(CA)^{\#}$ exists. Then $AB\sim CA$.
\end{proof}

\begin{cor}\label{cor2}
Let $A, B, C\in R^{n\times n}$ and $ABA=ACA$. If $R_{r}(A)=R_{r}(AB)$, $R_{r}(B)=R_{r}(BA)$, then $AB\sim CA$.
\end{cor}
\begin{proof}
Since $R_{r}(A)=R_{r}(AB)=AR_{r}(B)=AR_{r}(BA)=R_{r}(ABA)$, by Corollary \ref{cor1}, $AB\sim CA$.
\end{proof}

\begin{thm}\label{thm2}
Let $A, B, C\in R^{n\times n}$ with $ABA=ACA$. If $R_{r}(AB)=R_{r}(ABA)$, $R_{r}(CA)=R_{r}(CAB)$, then $AB\sim CA$.
\end{thm}
\begin{proof}
By the column space of a matrix, we have the formulas
$R_{r}(AB)=R_{r}(ABA)=R_{r}(ACA)=AR_{r}(CA)=AR_{r}(CAB)=R_{r}(ACAB)=R_{r}(ABAB)$
and
$R_{r}(CA)=R_{r}(CAB)=CR_{r}(AB)=CR_{r}(ABAB)=CR_{r}(ACAB)=R_{r}(CACAB)\subseteq R_{r}(CACA)$.
Obviously, $R_{r}(CA)\supseteq R_{r}(CACA)$. by \cite[Theorem 2.1]{CL}, $AB$ and $CA$ are group invertible.
Then the conclusion follows from Theorem \ref{thm1}.
\end{proof}
As an immediate consequence, we have

\begin{cor}\label{cor3}
Let $A, B, C\in R^{n\times n}$ and $ABA=ACA$. If $R_{r}(A)=R_{r}(AC)=R_{r}(ABA)$, then $AC\sim BA$.
\end{cor}
\begin{proof}
Since $R_{r}(A)=R_{r}(AB)=AR_{R}(B)=AR_{R}(BA)=R_{R}(ABA)$, by Corollary \ref{cor1}, $AB\sim CA$.
\end{proof}

\begin{exam}
Let
$A=\begin{bmatrix}
   1 &  1  \\
   0 & -1  \\
   \end{bmatrix}$,
$B=\begin{bmatrix}
   1  &  1   \\
   0  &  0   \\
   \end{bmatrix}$,
$C=\begin{bmatrix}
  1  &  -1   \\
  0  &   0   \\
  \end{bmatrix}$
$\in {\mathbb{C}}^{2\times 2}$, then $ABA=ACA$. Since $AB$ and $CA$ are idempotent, $(AB)^{\#}=AB$ and $(CA)^{\#}=CA$.
By virtue of Theorem \ref{thm1}, we have $AB\sim CA$. In fact,
$AB=P(CA)P^{-1}$, where
$P=\begin{bmatrix}
     1 & 1 \\
     0 & 1 \\
   \end{bmatrix}$.
\end{exam}

\section{Extensions}

In this section, we turn to the similarity of group inverses. We have
\begin{thm}\label{thm2}
Let $R$ be a B\'ezout domain, and let $A,B,C\in R^{n\times n}$ with $ABA=ACA$. If $AB$ and $CA$ are group invertible,
then $(AB)^{\#}$ is similar to $(CA)^{\#}$.
\end{thm}
\noindent
\begin{proof}
As in the proof of Theorem \ref{thm1},
$$A_1
\begin{bmatrix}
B_1 & B_2 \\
0   & 0   \\
\end{bmatrix}
=S^{-1}
\begin{bmatrix}
I  &  0   \\
G  &  I   \\
\end{bmatrix}
\begin{bmatrix}
C_4 D & 0   \\
0     & 0   \\
\end{bmatrix}
\begin{bmatrix}
 I  &  0   \\
-G  &  I   \\
\end{bmatrix}S,$$
similarly, there exists $G'$ such that
$$\begin{bmatrix}
B'_1 & B'_2 \\
 0   & 0   \\
\end{bmatrix}A'_1
=
R'
\begin{bmatrix}
I    &    -G'   \\
0    &    I     \\
\end{bmatrix}
\begin{bmatrix}
D'C'_4 & 0   \\
0      & 0   \\
\end{bmatrix}
\begin{bmatrix}
I    &    G'    \\
0    &    I     \\
\end{bmatrix}
R'^{-1}.$$
So
$$AB=H
\begin{bmatrix}
C_4 D & 0   \\
0     & 0   \\
\end{bmatrix}
H^{-1} \ {\text{and}} \
CA=K
\begin{bmatrix}
D'C'_4 & 0   \\
0      & 0   \\
\end{bmatrix}
K^{-1}$$
where
$$H=P
\begin{bmatrix}
 I  &  -E   \\
 0  &   I   \\
\end{bmatrix}
\begin{bmatrix}
 S^{-1}\begin{bmatrix}
             I  &   0   \\
             G  &   I   \\
       \end{bmatrix}  &               \\
                      &      I        \\
\end{bmatrix}$$
and
$$K=Q'^{-1}
\begin{bmatrix}
I    &    0    \\
F    &    I     \\
\end{bmatrix}
\begin{bmatrix}
R'\begin{bmatrix}
I    &    -G'    \\
0    &    I      \\
\end{bmatrix}          &          \\
                       &    I     \\
\end{bmatrix}.$$
It is easy to check that $AB$ and $CA$ are group invertible, and
$$(AB)^{\#}=H
\begin{bmatrix}
D^{-1}C_{4}^{-1} & 0   \\
0                & 0   \\
\end{bmatrix}
H^{-1},$$
$$(CA)^{\#}=K
\begin{bmatrix}
{C'_{4}}^{-1}{D'}^{-1} & 0   \\
0                  & 0   \\
\end{bmatrix}
K^{-1}.$$
In view of the similarity of
$\begin{bmatrix}
C_4D   &    0   \\
0      &    0   \\
\end{bmatrix}$
 and
$\begin{bmatrix}
D'C'_4   &    0   \\
0        &    0   \\
\end{bmatrix}$
which has been proved in Theorem \ref{thm1}, we derive that $(AB)^{\#}$ and $(CA)^{\#}$ are similar.
\end{proof}

\begin{cor}\label{cor4}
Let $A, B, C\in \mathbb{C}^{n\times n}$ with $ABA=ACA$. If $(AC)^{\#}$ and $(BA)^{\#}$ exist,
then $(AC)^{\#}$ is similar to $(BA)^{\#}$.
\end{cor}
\begin{proof}
This is obvious by Theorem \ref{thm2}.
\end{proof}

\begin{cor}\label{cor5}
Let $R$ be a B\'ezout domain, and let $A,B,C\in R^{n\times n}$ with $ABA=ACA$. If $AB$ and $CA$ are group invertible,
then $(AB)(AB)^{\#}$ is similar to $(CA)(CA)^{\#}$.
\end{cor}
\noindent
\begin{proof}
From the proof of Theorem \ref{thm2},
$$\begin{array}{rcl}
\vspace{2mm}
AB(AB)^{\#}   &   =   &
H\begin{bmatrix}
  C_4D   &          \\
         &      0   \\
  \end{bmatrix}H^{-1}H
  \begin{bmatrix}
  (C_4D)^{-1}   &          \\
                &      0   \\
  \end{bmatrix}H^{-1}          \\
   &   =   &
H\begin{bmatrix}
  I      &          \\
         &      0   \\
  \end{bmatrix}H^{-1},
\end{array}$$
$$\begin{array}{rcl}
\vspace{2mm}
CA(CA)^{\#}    &    =    &
K\begin{bmatrix}
  D'C'_4   &          \\
             &      0   \\
  \end{bmatrix}K^{-1}K
  \begin{bmatrix}
  (D'C'_4)^{-1}   &          \\
                    &      0   \\
  \end{bmatrix}K^{-1}               \\
  &   =    &
K\begin{bmatrix}
  I      &          \\
         &      0   \\
  \end{bmatrix}K^{-1}.
\end{array}$$
This implies that  $(AB)(AB)^{\#}$ and $(CA)(CA)^{\#}$ are similar.
\end{proof}

Cline proved that $ba$ is Drazin invertible if $ab$ has Drazin inverse. In this case, $(ba)^D=b[(ab)^D]^2a$. We now derive

\begin{thm}\label{thm3}
Let $R$ be a B\'ezout domain, and let $A,B,C\in R^{n\times n}$ with $ABA=ACA$. If $AB$ has Drazin inverse, then there exists $k\in {\Bbb N}$ such that
$(AB)^s$ is similar to $(CA)^s$ for any $s\geq k$.
\end{thm}
\begin{proof}
Suppose that $AB$ has Drazin inverse with $ind(AB)=k$.
By \cite[Theorem 2.7]{ZZ1},
$CA$ is Drazin invertible with $ind(CA)\leq k+1$, and $(CA)^D=C[(AB)^{D}]^2 A$.    \\
\indent
Set $s>k$. Since $AB(AB)^D$ is idempotent, we have
$$(AB)^s [(AB)^D]^s=AB(AB)^D=[(AB)^D]^s (AB)^s,$$
$$([(AB)^D]^s)^2(AB)^s=[(AB)^D]^{s+1}AB=[(AB)^D]^{s},$$
$$[(AB)^{s}]^2[(AB)^D]^s=(AB)^{s+1}(AB)^D=(AB)^{s}.$$
Accordingly, $(AB)^s$ is group invertible and $[(AB)^s]^{\#}=[(AB)^D]^s$. Similarly, $(CA)^s$ is group invertible and $[(CA)^s]^{\#}=[(CA)^D]^s$.    \\
\indent
Let $B'=B(AB)^{s-1}$ and $C'=(CA)^{s-1}C$. Then $AB'A=AC'A$, where $AB'$ and $C'A$ are group invertible.
Therefore, the result of theorem follows by Theorem \ref{thm1}.
\end{proof}

\begin{cor}\label{cor6}
Let $A,B\in M_{n}(\mathbb{C})$. Then there exists $k\in {\Bbb N}$ such that
$(AB)^s$ is similar to $(BA)^s$ for any $s\geq k$.
\end{cor}
\begin{proof} This is obvious by choosing $B=C$ in Theorem \ref{thm3}.\end{proof}

\begin{cor}\label{cor7}
Let $R$ be a B\'ezout domain, and let $A,B,C\in R^{n\times n}$ with $ABA=ACA$. If $AB$ and $CA$ are group invertible,
then $(AB)^2(AB)^{D}$ is similar to $(CA)^2(CA)^{D}$.
\end{cor}
\begin{proof}
Since AB and CA are group invertible, $(AB)^D=(AB)^{\#}, (CA)^D=(CA)^{\#}$. By the proof of Corollary \ref{cor5},
$$(AB)^2(AB)^D =
H\begin{bmatrix}
  C_4D   &          \\
         &      0   \\
  \end{bmatrix}H^{-1},$$
$$(CA)^2(CA)^D   =
K\begin{bmatrix}
  D'C'_4   &          \\
             &      0   \\
  \end{bmatrix}K^{-1}.$$
According to the proof of Theorem \ref{thm1},
$\begin{bmatrix}
C_4D   &    0   \\
0      &    0   \\
\end{bmatrix}$
 and
$\begin{bmatrix}
D'C'_4   &    0   \\
0        &    0
\end{bmatrix}$
are similar. Then the similarity of $(AB)^2(AB)^D$ and $(CA)^2(CA)^D$ follows.
\end{proof}


\vskip10mm\begin{thebibliography}{99}

\bibitem{CL} C. Cao and J. Li, Group inverses for matrices over a B\'ezout domain, {\it Electronic J. Linear Algebra}, {\bf 18}(2009), 600--612.

\bibitem{COHN} P. M. Cohn, Free rings and their relations, Second edition, Academic Press, 1985.

\bibitem{F} H. Flanders, Elementary diviors of $AB$ and $BA$, {\it Proc. Amer. Math. Soc.}, {\bf 2}(1951), 871--874.

\bibitem{HARTWIG} R. Hartwig, On a theorem of Flanders, {\it Proc. Amer. Math. Soc.}, {\bf 85}(1982), 310--312.

\bibitem{HUANG} L. Huang, Geometry of matrices over ring, Science Press, Beijing, 2007.

\bibitem{MD} N. Mihallovi\'c  and D. C. Djordjevi\'c, On group invertibility in rings, {\it Filomat}, {\bf 33}(2019), 6141--6150.

\bibitem{SHENG} Y. Sheng, Y. Ge, H. Zhang and C. Cao, Group inverses for a class of $2\times 2$ block matrices over rings, {\it Applied Math. Comput.}, {\bf 219}(2013), 9340--9346.

\bibitem{ZHANG} K. Zhang and C. Bu, Group inverses of matrices over right Ore domains, {\it Applied Math. Comput.}, {\bf 218}(2012), 6942--6953.

\bibitem{ZHAO} C. Zhao and J. Li, Group inverses for a class of $2\times 2$ block matrices over a B\'ezout domain, {\it Proceeding of the Sixth International Conference of Matrices and Operators}, Chengdu, China, July 8--11,2011, 90--93.

\bibitem{ZZ1} Q. Zeng and H. Zhong, New results on common properties of the products AC and BA, {\it J. Math. Anal. Appl.}, {\bf 427}(2015), 830--840.

\bibitem{ZZ2} Q. Zeng and H. Zhong, Common properties of bounded linear operators AC and BA: Spectral theory, {\it Math. Nachr.}, {\bf 287}(2014), 717--725.

\bibitem{ZZ3} Q. Zeng and H. Zhong, Common properties of bounded linear operators AC and BA: Local spectral theory, {\it J. Math.Anal.Appl.}, {\bf 414}(2014), 553--560.
\end{thebibliography}
\end{document}